\documentclass[10pt]{amsart}
\usepackage{graphicx}
\usepackage{caption}
\usepackage{subfigure}

\usepackage{amsthm}

 \usepackage{amsrefs}
\usepackage{mathtools} 
\usepackage{stmaryrd} 
\usepackage{framed} 
\usepackage{enumerate} 
\usepackage{tikz}
\usetikzlibrary{calc}
\usepackage{MnSymbol}
\usepackage{amsmath,mathtools}
\usepackage[noadjust]{marginnote}
\usepackage{color} 
\newcounter{mnote}
\newtheorem{thm}{Theorem}[section]
\newtheorem{prop}[thm]{Proposition}
\newtheorem{lem}[thm]{Lemma}
\newtheorem{cor}[thm]{Corollary}

\theoremstyle{definition}
\newtheorem{definition}[thm]{Definition}

\theoremstyle{definition}
\newtheorem{remark}[thm]{Remark}
 
\numberwithin{equation}{section}

\newcommand{\uh}{\hat{u}}
\newcommand{\vh}{\hat{v}}
\newcommand{\ko}{\kappa_0}
\newcommand{\Rt}[1]{\tilde{{R}}_{#1}}

\def\k0{\kappa_0}

\def\cA{\mathcal{A}}
\def\CCG{\mathcal{C}^2\mathcal{G}}

\begin{document}

\title[On whether zero is in the global attractor of the 2D-NSEs]{On
  whether zero is in the global attractor of the 2D Navier-Stokes equations}
\date{\today}
\author[C. Foias]{Ciprian Foias$^{1}$}
\address{$^1$Department of Mathematics\\
Texas A\&M University\\ College Station, TX 77843}
\author[M.S. Jolly]{Michael S. Jolly$^{2}$}
\address{$^2$Department of Mathematics\\
Indiana University\\ Bloomington, IN 47405}
\author[Y. Yang]{Yong Yang$^{1,\dagger}$}
\author[B. Zhang]{Bingsheng Zhang$^{1}$}
\address{$\dagger$ corresponding author}
\email[C. Foias]{foias@math.tamu.edu}
\email[M. S. Jolly]{msjolly@indiana.edu}
\email[Y. Yang] {yytamu@math.tamu.edu}
\email[B. Zhang] {bszhang@math.tamu.edu}
\thanks{This work was supported in part by NSF grant numbers DMS-1109638, and DMS-1109784}

%
%
\begin{abstract}
The set of nonzero external forces for which the zero function is in the 
global attractor of the 2D Navier-Stokes equations is shown to be meagre in a Fr\'echet topology.
A criterion in terms of a Taylor expansion in complex time is used to characterize 
the forces in this set.   This leads to several relations between certain Gevrey subclasses of $C^{\infty}$
and a new upper bound for a Gevrey norm of solutions in the attractor, valid in the strip of analyticity in time.
\end{abstract}
\maketitle

\section{Introduction}
A challenge posed by P. Constantin \cite{PC03} is to find a simple proof that zero is
in the global attractor $\mathcal{A}_g$ of the 2D Navier-Stokes equations (NSE) if and only if the external force $g$ is zero. A related and perhaps equally challenging problem is to find sharp lower bounds on the energy in cases where we know $0\nin\mathcal{A}_g$. A 
bound which is probably far from sharp can be found in \cite{DFJ05}.  Such a lower bound 
can have implications for turbulent flow because a direct cascade of energy is indicated
by a large quotient of average enstrophy to average energy \cite{FJMR}.   

In this
paper, we show that the set of nonzero forces for which $0 \in \cA_g$ is meagre (of the first Baire 
category in a Fr\'echet topology (see Theorem \ref{t73})). In doing so we establish several relations between certain Gevrey subclasses of $C^{\infty}$
(see Theorem \ref{t52}, Corollary \ref{c54}).  We also prove a new upper bound for a Gevrey norm
of solutions in the attractor, valid for all $\zeta$ in the strip $\mathcal{S}(\delta)$
of time-analyticity (see Theorem \ref{t63}).  Moreover, by using complex time analytic techniques
from \cite{FJK96}, we present a concrete criterion that is both
sufficient and necessary for $0 \in \mathcal{A}_g$. 
We demonstrate the use of this criterion to prove that zero is not in the global attractor
in the particular case of Kolmogorov forcing (where $g$ is in an
eigenfunction of the Stokes operator).


\section{General preliminaries}

We consider the Navier-Stokes equations with periodic boundary conditions in $\Omega=[0, L]^2$
\begin{equation}\label{NSEs}
\begin{array}{rl}

\frac{\partial u}{\partial t}-\nu \Delta u + (u \cdot \nabla) u + \nabla p = F,\\

\nabla \cdot u=0, \\

u(x,0)=u_{0}(x), \\

\int_{\Omega} u\,dx=0,\qquad \int_{\Omega} F\, dx=0.
\end{array}
\end{equation}
where $u:\mathbb{R}^2\rightarrow \mathbb{R}^2$, $p:\mathbb{R}^2 \rightarrow \mathbb{R}$ are unknown $\Omega$-periodic functions, and $\nu>0$ is the viscosity of the fluid, $L>0$ is the period, $p$ is the pressure, and $F$ is the ``body'' force (see \cite{T95}, \cite{CF89}, \cite{T97} for more details). The phase space $H$ is defined as the subspace of $[L^2(\Omega)]^2$ consisting of all elments in the closure of the set of $\mathbb{R}^2$-valued trigometric polynomials $v$ satisfying
\begin{equation*}
\nabla \cdot v=0  \ \text{and}   \int_{\Omega}v(x)dx=0.
\end{equation*}
The scalar product in $H$ is taken to be
\begin{equation*}
(u,v)=\int_{\Omega}u(x)\cdot v(x)dx
\end{equation*}
with associated norm $|u|=(u,u)^{1/2}$.

Let $\mathcal{P}:[L^2(\Omega)]^2 \rightarrow [L^2(\Omega)]^2$ be the orthogonal projection (called the Helmholtz-Leray projection) with range $H$, and  define the Stokes operator $A=-\mathcal{P}\Delta$ (=$-\Delta$, under periodic boundary conditions), which is positive, self-adjoint with a compact inverse. As a consequence, the real Hilbert space $H$ has an orthonormal basis $\{\omega_j\}^{\infty}_{j=1}$ eigenfunctions of $A$, namely, $A\omega_j=\lambda_j \omega_j$ with $0<\lambda_{1}=\big(\frac{2\pi}{L}\big)^2\leq \lambda_{2}\leq \lambda_{3}< \cdots$. The powers $A^{\sigma}$ are defined by 
\begin{equation*}
A^{\sigma}v=\sum_{j=1}^{\infty} \lambda_j^{\sigma}(v,\omega_j)\omega_j,  \quad \sigma \in \mathbb{R},
\end{equation*}
where $(\cdot, \cdot)$ is the $L^2-$ scalar product. The domain of $A^{\sigma}$ is denoted by $\mathcal{D}(A^\sigma).$

The system  \eqref{NSEs} can be written as a differential equation 
\begin{equation}
\label{35}
\frac{du}{dt}+\nu Au+B(u,u)=g, \quad u \in H.
\end{equation}
where the bilinear operator $B$ and the driving force $g$ are defined as $B(u,v)=\mathcal{P}((u\cdot\nabla)v)$ and $g=\mathcal{P}F$, respectively.

 Under periodic boundary conditions, we may express an element $u \in H $ as a Fourier series expansion 
\begin{equation*}
u(x)=\sum_{k \in \mathbb{Z}^2 \setminus \{0\}}\hat{u}(k)e^{i\kappa_{0}k\cdot x},
\end{equation*}
where 
$\kappa_{0}=\lambda_{1}^{1/2}=\frac{2\pi}{L},\ \hat{u}(0)=0,\ (\hat{u}(k))^{\ast}=\hat{u}(-k),\ k\cdot\hat{u}(k)=0.$
Parseval's identity reads as
\begin{equation*}
|u|^2=L^2\sum_{k\in\mathbb{Z}^2\setminus \{0\}}\hat{u}(k)\cdot \hat{u}(-k)=L^2\sum_{k\in\mathbb{Z}^2\setminus\{0\}}|\hat{u}(k)|^2,
\end{equation*}
or more generally
\begin{equation*}
(u,v)=L^2\sum_{k\in\mathbb{Z}^2\setminus \{0\}}\hat{u}(k)\cdot \hat{v}(-k).
\end{equation*}

The following inequalities will be repeatedly used in this paper
\begin{equation}
\label{Poineq}
\kappa_{0}|u|\leq|A^{1/2}u|, \text{for} \ u\in \mathcal{D}(A^{1/2}),
\end{equation}
\begin{equation}
\label{Laineq}
|u|_{L^{4}(\Omega)}\leq \sqrt{2}c_{L}|u|^{1/2}|A^{1/2}u|^{1/2}, \text{for} \ u\in \mathcal{D}(A^{1/2}),
\end{equation}
\begin{equation}
\label{Agineq}
|u|_{\infty}\leq c_{A}|u|^{1/2}|Au|^{1/2}, \text{for} \ u\in \mathcal{D}(A).\\
\end{equation}
known respectively as the Poincar\'{e}, Ladyzhenskaya and Agmon inequalities. Both $c_{L}$ and $c_{A}$ are absolute constants.

We recall that the global attractor $\mathcal{A}$ of the NSE is the collection of all elements $u_{0}$ in $H$ for which there exists a solution $u(t)$ of the NSE, for all $t\in \mathbb{R}$, such that $u(0)=u_{0}$ and $\sup_{t\in \mathbb{R}}|u(t)|<\infty$.

To give another definition of $\mathcal{A}$, we need to recall several concepts. First, as is well-known, for any $u_{0}\in H$, $f\in H$, there exists a unique continuous function $u$ from $[0,\infty)$ to $H$ such that $u(0)=u_{0}$, $u(t)\in \mathcal{D}(A)$, $t\in (0,\infty)$, and $u$ satisfies the NSE for all $t\in(0,\infty)$. Therefore, one can define the map $S(t):H\rightarrow H$ by
\begin{equation}
\label{311}
S(t)u_{0}=u(t),\\
\end{equation}
where $u(\cdot)$ is as above. Since $S(t_{1})S(t_{2})=S(t_{1}+t_{2})$, the family $\{S(t)\}_{t\geq 0}$  is called the ``solution'' semigroup. Furthermore, a compact set $\mathcal{B}$ is called absorbing if for any bounded set $\tilde{\mathcal{B}}\subset H$ there is a time $\tilde{t}=\tilde{t}(\tilde{\mathcal{B}})\geq 0$ such that $S(t)\tilde{\mathcal{B}}\subset \mathcal{B}$ for all $t \geq \tilde{t}$. The attractor can be now defined by the formula
\begin{equation}
\label{312}
\mathcal{A}=\bigcap_{t\geq 0} S(t)\mathcal{B},
\end{equation}
where $\mathcal{B}$ is any absorbing compact subset of $H$.

We now consider the NSE with complexified time and the corresponding solutions in $H_{\mathbb{C}}$ as in \cite{FHN07} and \cite{FJK96}. We recall that 
\begin{equation*}
H_{\mathbb{C}}=\{u+iv: u, v\in H\},
\end{equation*}
and that $H_{\mathbb{C}}$ is a Hilbert space with respect to the following inner product
\begin{equation*}
(u+iv,u'+iv')_{H_{\mathbb{C}}}=(u,u')_{H}+(v,v')_{H}+i[(v,u')_{H}-(u,v')_{H}],
\end{equation*}
where $u,u',v,v' \in H$. The extension $A_{\mathbb{C}}$ of $A$ is given by 
\begin{equation*}
A_{\mathbb{C}}(u+iv)=Au+iAv,
\end{equation*}
for $u,v \in \mathcal{D}(A)$; thus $\mathcal{D}(A_{\mathbb{C}})=\mathcal{D}(A)_{\mathbb{C}}$. Similarly, $B(\cdot,\cdot)$ can be extended to a bounded bilinear operator from $\mathcal{D}(A_{\mathbb{C}}^{1/2}) \times \mathcal{D}(A_{\mathbb{C}})$ to $H_{\mathbb{C}}$ by the formula
\begin{equation*}
B_{\mathbb{C}}(u+iv,u'+iv')=B(u,u')-B(v,v')+i[B(u,v')+B(v,u')],
\end{equation*}
for $u,v \in \mathcal{D}(A^{1/2}),u',v' \in \mathcal{D}(A)$.

The Navier-Stokes equation in complex time is defined as
\begin{equation}
\label{cnse}
\frac{du(\zeta)}{d\zeta}+\nu A_{\mathbb{C}}u(\zeta)+B_{\mathbb{C}}(u(\zeta),u(\zeta))=g,
\end{equation}
where $\zeta \in \mathbb{C}$, $u(\zeta) \in H_{\mathbb{C}}$ and $\frac{du(\zeta)}{d\zeta}$ denotes the derivatives of $H_{\mathbb{C}}$-valued analytic function $u(\zeta)$.

\section{Specific preliminaries}

In this section, we first recall the definition of the class $\mathcal{C}(\sigma)$ introduced in our previous paper \cite{FJLRYZ13}. We also collect the properties regarding the class $\cup_{\sigma>0}\mathcal{C}(\sigma)$.
The class $\mathcal{C}(\sigma)$ is defined to be a  subset of $C^{\infty}([0,L]^2)\cap H$  for which every element $u\in  \mathcal{C}(\sigma)$ has a specified growth rate for the powers of the operator $A$ applied to $u$
\begin{align}\label{Csigdef}
\mathcal{C}(\sigma):=\{u\in C^{\infty}([0,L]^2)\cap H:\ \exists \  c_0 = c_0(u)\in \mathbb{R}, \frac{|A^{\frac{\alpha}{2}}u|^2}{\nu^2 \kappa_0^{2\alpha}}\leq c_0e^{\sigma \alpha^2}\;,
 \ \forall \alpha \in \mathbb{N}\}\;.
\end{align}

In this definition we allow only $\alpha \in \mathbb{N}$; however, as shown in Section 11 of \cite{FJLRYZ13}, we could actually extend this definition to allow $\alpha$ to take any real numbers without changing the class $\mathcal{C}(\sigma)$. We stress that the constant $c_0 \in \mathbb{R}$ in the definition of the class $\mathcal{C}(\sigma)$ depends on $u$.

To make our presentation more self-contained, we include some of the
relevant results from \cite{FJLRYZ13}. The first result gives some
consequences of zero belonging to the attractor.
\begin{thm}
If $0\in \mathcal{A}$, then both the attractor $\mathcal{A}$ and the force $g$ will be in the class $\mathcal{C}(\sigma)$ ; namely
\begin{equation}
\label{bf831}
0\in \mathcal{A} \Rightarrow \mathcal{A}\subset \mathcal{C}(\sigma_0)\quad 
\textrm{and} \quad g\in \mathcal{C}(\sigma_1),
\end{equation}
for some $\sigma_1>\sigma_0>0$, where $\sigma_0$ and $\sigma_1$ both depend on the force $g$.
\end{thm}

In particular, we have the following estimates
\begin{thm}
\label{thm33}
\label{t64}
If $0\in\mathcal{A}$, then there exist fixed positive constants $\tilde{R}_1$, $\tilde{R}_2$, $\tilde{R}_3$, $C(g)$, $\beta_1$, $\beta_2$, $\delta_1$,  $\delta_2$, $\delta_3$ such that the analytic extension of any solution $u(t), \ t\in\mathbb{R}$ in $\mathcal{A}$ satisfies for
any $\alpha \in \mathbb{N}$
	\begin{equation*}
        |A^{\frac{\alpha+1}{2}}u(\zeta)|\leq \Rt{\alpha+1}\nu\ko^{\alpha+1},\quad\forall\  \zeta\in \mathcal{S}(\delta_\alpha) :=\{\zeta\in\mathbb{C}:|\Im(\zeta)|<\delta_\alpha\},
    \end{equation*}
where for $\alpha>3$,
	\begin{equation*}
		\Rt{\alpha+1}^2\leq C(g){\beta_1}^{4^{\alpha+1}}\beta_2^{(\alpha+1)^2+\frac{9}{2}(\alpha+1)}\;,
		\qquad \delta_\alpha=\delta_3\;.
	\end{equation*}
\end{thm} 

For the sake of completeness, explicit expressions for $\tilde{R}_1$, $\tilde{R}_2$, $\tilde{R}_3$, $C(g)$, $\beta_1$, $\beta_2$, $\delta_1$,  $\delta_2$, $\delta_3$ are recalled in the Appendix.
The next result from \cite{FJLRYZ13} merely states a simple hierarchy of the spaces
$\mathcal{C}(\sigma), \sigma \in \mathbb{R}^+$.
\begin{prop}
\label{increasing}
For the family of classes $\{\mathcal{C}(\sigma)\}_{\sigma>0}$, we have,
\begin{equation*}
\mathcal{C}(\sigma_1)  \subseteq   \mathcal{C}(\sigma_2), \quad \forall \sigma_1< \sigma_2.
\end{equation*}
\end{prop}

The union of the classes $\mathcal{C}(\sigma)$ is a proper subset of $C^{\infty}$.
\begin{thm}
\label{th11-1}
$\bigcup_{\sigma>0}\mathcal{C}(\sigma)\subsetneqq C^{\infty}([0,L]^2; \mathbb{R}^2)\cap H$
\end{thm}

The following ``all for one, one for all" law states that the attractor cannot be 
only partially contained in the union of these classes. 
\begin{thm}
\label{theorem121}
If $\mathcal{A}\cap \cup_{\sigma>0}\mathcal{C}(\sigma)\neq \emptyset$, then $\mathcal{A}\subset \cup_{\sigma>0}\mathcal{C}(\sigma)$.
\end{thm}

\section{Constantin-Chen Gevrey classes}

In this section, we give the definition for the general Constantin-Chen Gevrey ($\CCG$) classes \cite{CHEN94}.

Given a function $\phi(\chi)$  with the following properties: 
\begin{align*}
&(1) \quad {\phi}'(\chi)>0, \\
& (2) \quad {\phi}''(\chi)<0,
\end{align*}
for all $\chi \in [1, \infty)$,  we define the general Constantin-Chen Gevrey ($\CCG$) class $E(\phi)$ as the collection of all  $u \in C^{\infty}([0,L]^2)\cap H$ for which $|e^{\phi(\kappa^{-1}_0 A^{\frac{1}{2}})}u|$ is finite, that is,
\begin{definition}
$E(\phi)=\{u \in H: |e^{\phi(\kappa^{-1}_0 A^{\frac{1}{2}})}u| < \infty \}$,
\end{definition}
where 
\begin{align*}
(e^{\phi(\kappa^{-1}_0 A^{\frac{1}{2}})}u\hat{)}(k):=e^{\phi(|k|)}\hat{u}(k),\quad \forall\, k \in \mathbb{Z}^2\setminus\{0\}.
\end{align*}

A typical example of a $\CCG$ class is
$\tilde{\phi}(\chi)=\beta \ln\chi$, for $\beta>0$. Actually, $E(\tilde{\phi}) = H\cap \mathcal{D}(A^{\beta/2})$. This typical $\CCG$ class is used in \cite{FJLRYZ13} to prove the following two estimates for the bilinear term $|(B(u,v),A^{\gamma}w)|$, where $u\in \mathcal{D}(A^{\frac{\gamma}{2}})$, $v\in \mathcal{D}(A^{\frac{\gamma+1}{2}})$, $w\in \mathcal{D}(A^{\gamma})$, and $\gamma>3$.

\begin{lem}
\label{lemma61}
Let $u\in \mathcal{D}(A^{\frac{\alpha}{2}})$, $v\in \mathcal{D}(A^{\frac{\alpha+1}{2}})$,$w\in \mathcal{D}(A^{\alpha})$, and $\alpha>3$, then
\begin{equation*}
|(B(u,v),A^{\alpha}w)|\leq 2^{\alpha}c_A \left(|u|^{1/2}|Au|^{1/2}|A^{\frac{1+\alpha}{2}}v|+|A^{\frac{\alpha}{2}}u||A^{1/2}v|^{1/2}|A^{3/2}v|^{1/2} \right)|A^{\frac{\alpha}{2}}w|.
\end{equation*}
If, moreover, $u\in \mathcal{D}(A^{\frac{\alpha}{2}})_{\mathbb{C}}$, $v\in \mathcal{D}(A^{\frac{\alpha+1}{2}})_{\mathbb{C}}$,$w\in \mathcal{D}(A^{\alpha})_{\mathbb{C}}$, then
\begin{equation*}
|(B(u,v),A^{\alpha}w)|\leq 2^{\alpha+3/2}c_A \left(|u|^{1/2}|Au|^{1/2}|A^{\frac{1+\alpha}{2}}v|+|A^{\frac{\alpha}{2}}u||A^{1/2}v|^{1/2}|A^{3/2}v|^{1/2} \right)|A^{\frac{\alpha}{2}}w|.
\end{equation*}
\end{lem}


\section{$\mathcal{C}(\sigma)$ and $E(\phi_b)$}

In this section we will investigate the relation between the class $\mathcal{C}(\sigma)$ and  $E({\phi}_b)$, where
\begin{equation} 
\phi_b(\chi)=b[\ln(\chi+e)]^2, \quad b>0.
\end{equation}
The main results are stated in Theorem \ref{t52} and Theorem \ref{c54}.

For convenience, we take the following notation 
\begin{definition}
For $b>0$, we define
$E_b := E(\phi_b)$, $E^bu := e^{\phi_b({\ko^{-1} A^{1/2}})}u$,  $|u|_b := |E^bu|$.
\end{definition}

In our previous paper \cite{FJLRYZ13}, we have obtained the following result.
\begin{thm} [see Remark 7.6 in \cite{FJLRYZ13}]
\label{t52}
If $v \in E_b$ for some $b>0$,
then
\begin{equation*}
v\in \mathcal{C}(\frac{1}{2b}).
\end{equation*}
\end{thm}

The ``reverse'' inclusion relation between the classes $E_b$ and $\mathcal{C}(\sigma)$  is given in Theorem \ref{c54}. 

\begin{prop}
\label{t53}
	If $u\in \mathcal{C}(\sigma)$ for some $\sigma > 0$, i.e., 
	\begin{equation}
		\exists\, c_0>0,\quad \text{s.t.}\quad |A^{\frac{\alpha}{2}}u|^2 \leq c_0e^{\sigma \alpha^2}(\nu\ko^\alpha)^2, \quad \forall \alpha\in \mathbb{N},
	\end{equation}
	 then for fixed $\epsilon\in[0, 1]$, there exists $b := \frac{1}{2^{4+2\epsilon}\sigma}$ such that 
	\begin{align}
           |e^{b[\ln(\ko^{-1} A^{1/2}+e)]^{1+\epsilon}}u| < \infty.
	\end{align}
	
	In particular, we have
	\begin{align}
    	 |e^{b[\ln(\ko^{-1} A^{1/2}+e)]^{1+\epsilon}}u|^2 \leq \frac{4}{3}c(\epsilon)|u|^2+c_0^{\frac{1}{3}}(c_1|A^{\frac{1}{2}}u|)^{\frac{2}{3}}\nu^{\frac{4}{3}}\ko^{-\frac{2}{3}},
	\end{align}
	where 
	\begin{align*}
		c(\epsilon) := e^{2b[\ln(e^2+e)]^{1+\epsilon}}, \quad c_1 =\sum_{m\geq 2} \frac{1}{e^{m}}=\frac{1}{e^2-e}.
	\end{align*}

\end{prop}


\begin{proof}
First, by the definition of $e^{\phi(A^{1/2})u}$, we have
 	\begin{align*}
 		|e^{b[\ln(\ko A^{1/2}+e)]^{1+\epsilon}}u|^2 =& \sum_{m=0}^{\infty}\sum_{e^m\leq |k| < e^{m+1}}e^{2b[\ln(|k|+e)]^{1+\epsilon}}|\uh (k)|^2\\
 		=& \sum_{m=0,1} + \sum_{m\geq 2}
 		=: I_1 + I_2.
	\end{align*}	 

For $I_1$,  	it is easy to see that
	\begin{align*}
		I_1 =& \sum_{m=0,1}\sum_{e^m\leq |k| < e^{m+1}}e^{2b[\ln(|k|+e)]^{1+\epsilon}}|\uh (k)|^2\\
		\leq& e^{2b[\ln(e^2+e)]^{1+\epsilon}}\sum_{k\in \mathbb{Z}^2\setminus\{0\}}|\uh (k)|^2\\
		=& e^{2b[\ln(e^2+e)]^{1+\epsilon}}|u|^2\\
		=& c(\epsilon)|u|^2<\infty,
	\end{align*}
while for $I_2$, using the definition of the class $\mathcal{C}(\sigma)$ and Young's inequality we can infer 
	\begin{align*}
	I_2 =& \sum_{m\geq 2}\sum_{e^m\leq |k| < e^{m+1}}e^{2b[\ln(|k|+e)]^{1+\epsilon}}|\uh (k)|^2\\
		=& \sum_{m\geq 2}\sum_{e^m\leq |k| < e^{m+1}} (|k|+e)^{2b[\ln(|k|+e)]^{\epsilon}}|\uh (k)|^2\\
		\leq& \sum_{m\geq 2}\sum_{e^m\leq |k| < e^{m+1}} (|k|+e)^{2b(m+2)^{\epsilon}}|\uh (k)|^2\\
		\leq& \sum_{m\geq 2}\sum_{e^m\leq |k| < e^{m+1}} |k|^{4b(m+2)^{\epsilon}}|\uh (k)|^2\\
		=& \sum_{m\geq 2}|A^{b(m+2)^{\epsilon}}(P_{m+1}-P_{m})u|^2\ko^{-4b(m+2)^\epsilon}\\
		\leq&\sum_{m\geq 2}|A^{2b(m+2)^{\epsilon}}u||(P_{m+1}-P_{m})u|\ko^{-4b(m+2)^\epsilon}\\
		\leq& \nu \sum_{m\geq 2}c_0^{1/2} e^{\frac{\sigma}{2}16b^2(m+2)^{2\epsilon}}|(P_{m+1}-P_{m})u| \\
		\leq&\nu \sum_{m\geq 2}c_0^{1/2} e^{2^{3+2\epsilon}\sigma b^2m^{2\epsilon}}|(P_{m+1}-P_{m})u|^{\frac{1}{2}}|(P_{m+1}-P_{m})u|^{\frac{1}{2}}\\
		\leq&\nu \left(\sum_{m\geq 2}c_0^2 e^{2^{5+2\epsilon}\sigma b^2m^{2\epsilon}}|(P_{m+1}-P_{m})u|^{2}\right)^{1/4} \left(\sum_{m\geq 2}|(P_{m+1}-P_{m})u|^{2/3}\right)^{3/4}\\
		=:& \nu I_{21}^{1/4}I_{22}^{3/4}
	\end{align*}

We now derive estimates for $I_{21}$ and $I_{22}$. For $I_{21}$, we obtain
\begin{align*}
		I_{21} =& \sum_{m\geq 2}c_0^2 e^{2^{5+2\epsilon}\sigma b^2m^{2\epsilon}}|(P_{m+1}-P_{m})u|^{2}\\
		\leq& \sum_{m\geq 0}c_0^2 e^{2^{5+2\epsilon}\sigma b^2m^{2\epsilon}}|(P_{m+1}-P_{m})u|^{2}\\
		=& \sum_{m\geq 0}c_0^2 e^{2^{5+2\epsilon}\sigma b^2m^{2\epsilon}}\sum_{e^m\leq |k| < e^{m+1}}|\uh(k)|^{2}\\
		\leq& c_0^2\sum_{m\geq 0} \sum_{e^m\leq |k| < e^{m+1}}e^{2^{5+2\epsilon}\sigma b^2[\ln(|k|+e)]^{2\epsilon}}|\uh(k)|^{2}\\
		\leq& c_0^2\sum_{m\geq 0} \sum_{e^m\leq |k| < e^{m+1}}e^{2^{5+2\epsilon}\sigma b^2[\ln(|k|+e)]^{1+\epsilon}}|\uh(k)|^{2},\\
	\end{align*}
since 
	\begin{equation*}
		2\epsilon \leq 1+\epsilon, \quad \text{i.e.,}\quad  \epsilon \leq 1.
	\end{equation*}

Defining $b$ as
	\begin{equation*}
		2^{4+2\epsilon}\sigma b = 1,  \quad \text{i.e.,}\quad  b = \frac{1}{2^{4+2\epsilon}\sigma},
	\end{equation*}
	we immediately get 
	\begin{align*}
		I_{21} \leq& c_0^2\sum_{m\geq 0} \sum_{e^m\leq |k| < e^{m+1}}e^{2b[\ln(|k|+e)]^{1+\epsilon}}|\uh(k)|^{2}\\
		=& c_0^2|e^{b[\ln(\ko A^{1/2}+e)]^{1+\epsilon}}u|^2.
	\end{align*}
	
For $I_{22}$, we set $v=A^{1/2}u$ and apply H\"older's inequality as follows:
	\begin{align*}
		I_{22} =& \sum_{m\geq 2} |(P_{m+1}-P_m)u|^{2/3}\\
		=& \sum_{m\geq 2} |A^{-1/2}(P_{m+1}-P_m)A^{1/2}u|^{2/3}\\
		=& \sum_{m\geq 2} \left(\sum_{e^m\leq |k|<e^{m+1}}\frac{1}{\ko^2 |k|^2}|\vh(k)|^2\right)^{1/3}\\
		\leq& \frac{1}{\ko^{2/3}}\sum_{m\geq 2} \frac{1}{e^{2m/3}} \left(\sum_{e^m\leq |k|<e^{m+1}}|\vh(k)|^2\right)^{1/3}\\
		\leq&\frac{1}{\ko^{2/3}}\left(\sum_{m\geq 2} \frac{1}{e^{m}}\right)^{2/3}\left(\sum_{m\geq 2} \sum_{e^m\leq |k|<e^{m+1}}|\vh(k)|^2\right)^{1/3}\\
		\leq& \left(c_1\frac{|A^{1/2}u|}{\ko}\right)^{2/3}<\infty.
	\end{align*}
		Therefore, we have 
	\begin{align*}
		|e^{b[\ln(\ko^{-1}A^{1/2}+e)]^{1+\epsilon}}u|^2 \leq& I_1+I_2\\
		\leq& I_1 + I_{21}^{1/4}I_{22}^{3/4}\nu \\
		=& I_1 + (|e^{b[\ln(A^{1/2}+e)]^{1+\epsilon}}u|^2)^{1/4}c_0^{1/2}\nu I_{22}^{3/4}\\
		\leq& I_1 + \frac{1}{4}|e^{b[\ln(A^{1/2}+e)]^{1+\epsilon}}u|^2+\frac{3}{4}c_0^{2/3}\nu^{4/3}I_{22},
	\end{align*}
and then 
	\begin{align*}
		|e^{b[\ln(\ko^{-1}A^{1/2}+e)]^{1+\epsilon}}u|^2 \leq& \frac{4}{3}I_1 + c_0^{2/3}I_{22}\\
		\leq& \frac{4}{3}c(\epsilon)|u|^2+(c_0c_1)^{2/3}\nu^{4/3}\ko^{-2/3}|A^{1/2}u|^{2/3}.
	\end{align*}		
\end{proof}
By the proceeding proposition, taking $\epsilon=1$ we  obtain the following.
\begin{thm}
\label{c54}
	If $u\in\mathcal{C}(\sigma)$, then $u\in E_b$, where $b := \frac{1}{64\sigma}$.
\end{thm}

In Corollary 7.5 of \cite{FJLRYZ13}, we have proved that if $0\in \mathcal{A}$, then $g\in\mathcal{C}(\frac{5}{2}\ln\beta_3)$, where $\beta_3$ is defined in Theorem \ref{t64}. Applying Theorem \ref{c54}, we obtain the following result.
\begin{cor}
\label{c55}
If $0\in \mathcal{A}$, then $g\in E_b$, where $b=\frac{1}{160\ln\beta_3}$.
\end{cor}

\section{The topological properties of the ``all for one, one for all law'' classes}
In this section we use the space $\mathcal{F} := \mathcal{C}^\infty\cap H$ with the Fr\'{e}chet topology defined by the following metric
\begin{equation}
	d(u, v) := \sum_{\alpha=1}^\infty \frac{1}{2^\alpha}\frac{|A^{\frac{\alpha}{2}}(u-v)|}{1+|A^{\frac{\alpha}{2}}(u-v)|}.
\end{equation}
Let
\begin{equation}
	E_{b,n}:=\{u\in E_b, |u|_b\leq n\}.
\end{equation}

\begin{lem}
$E_{b,n}$ is nowhere dense in $(\mathcal{F}, d)$.
\end{lem}
\begin{proof}
	First, we prove $i_b: E_{b,n}\to (\mathcal{F}, d)$ is compact. Clearly, for all $\alpha\in\mathbb{N}$, there exist a constant $c_\alpha$ such that 
	\begin{equation}
	\label{e29}
		|A^{\frac{\alpha}{2}}u|\leq c_\alpha|u|_b.
	\end{equation}
	For any sequence $\{u_n\}\subset E_{b,n}$, we have that $\{|A^{\frac{\alpha}{2}}u_n|\}$ is bounded by (\ref{e29}). Therefore, there exists a subsequence $\{u_{n_m}\}$ which is convergent in $\mathcal{D}(A^{\frac{\alpha-1}{2}})$. Since this is true for any fixed $\alpha\in\mathbb{N}$, by the diagonal process, we obtain a subsequence, denoted by the same notation $\{u_{n_m}\}$ for convenience, which is convergent for any $\alpha\in\mathbb{N}$. Hence it is convergent in $\mathcal{C}^\infty$ with the metric $d(\cdot,\cdot)$. Therefore, $i_b$ is compact. Then, it follows that $\overline{i_b(E_{b,n})}$ is compact in $(\mathcal{F}, d)$.
	
	Secondly, suppose $i_b(E_{b,n})$ is not nowhere dense. Then there exists a ball $\overline{B(x_0, \epsilon)}\subset \overline{i_b(E_{b,n})}$. Clearly, $\overline{B(x_0, \epsilon)}$ is compact.
	
	If $x_0=0$, this contradicts the extension of the classical Riesz's Lemma for normed spaces to locally convex topological vector space,  since $\mathcal{C}^\infty$ is infinite-dimensional space.
	
	If $x_0\neq 0$, consider a convex open neighborhood $N_{x_0}\subset \overline{B(x_0, \epsilon)}$. We have that $\overline{N_{x_0}}$ and $-\overline{N_{x_0}}$ both are compact and convex. Let $f: \overline{N_{x_0}}\times (-\overline{N_{x_0}})\ni (x_1, x_2)\mapsto \frac{x_1+x_2}{2} \in \mathcal{C}^\infty\cap H$. Clearly, $f$ is continuous. Therefore the range $R(f)$ is compact. Since $\frac{1}{2}(\overline{N_{x_0}}+(-\overline{N_{x_0}}))$ is an open neighborhood of $0$ in $R(f)$. For the same reason as above, we get a contradiction. 	
\end{proof}

Due to the above lemma, it follows that
\begin{thm}
\label{c62}
	$\cup_{m=1}^\infty \cup_{n=1}^\infty E_{\frac{b}{m},n}$ is of first (Baire) category in $(\mathcal{F}, d)$.
\end{thm}

From the above theorem and Corollary \ref{c55}, we have that the conjecture that $g\neq 0$ implies $0\notin \mathcal{A}_g$ is ``almost'' true in the following sense.
\begin{thm}
\label{t73}
	If $0\in \mathcal{A}$, then $g\in E_b$ ($b$ is defined in Corollary \ref{c55}) where $E_b$ has the property that $E_b = \cup_{n=1}^\infty E_{b,n}$ is of first (Baire) category in $(\mathcal{F}, d)$.
\end{thm}

\section{Dynamical properties of the NSE in $E_b$}
This section is devoted to getting a new estimate for solutions in the global attractor with the norm $|\cdot|_b$ in the strip $\mathcal{S}(\delta)$, where
\begin{equation}
		\mathcal{S}(\delta):=\{\zeta\in\mathbb{C}:|\Im(\zeta)|<\delta\}.
	\end{equation}
First, for the nonlinear term $B(\cdot, \cdot)$ we need the following estimate. 
\begin{lem}
\label{l55}
If $u\in E_b$, then $B(u,u)\in E_b$ and 
\begin{align}
\label{e55}
	|E^bB(u,u)| \leq \frac{c_3c_A}{\ko^{2b\ln 2}}&\left(|A^{b\ln2}u|^{1/2}|A^{1+b\ln2}u|^{1/2}|A^{1/2}E^bu|\right.\\
	&+\left.|A^{1/2+b\ln2}u|^{1/2}|A^{3/2+b\ln2}u|^{1/2}|E^bu|\right), \nonumber
\end{align}
where  
\begin{equation}
 c_3 = e^{b(\ln2)^2}(1+e)^{2b\ln2}.
\end{equation}
\end{lem}
\begin{proof}
By definition, for any $w\in E_b$, we have
	\begin{align*}
		I :=& (E^bB(u,u), w) = (B(u,u), E^bw)\\
		  =& L^2\sum_{\substack{h,j,k\in\mathbb{Z}^2\setminus\{0\}\\h+j+k=0}}(\hat{u}(h)\cdot j)(\hat{u}(j)\cdot\hat{w}(k))e^{b(\ln(|k|+e))^2}\\
		  =& L^2\sum_{\substack{h,j,k\in\mathbb{Z}^2\setminus\{0\}\\h+j+k=0\\|h|\leq|j|}}\cdots +L^2\sum_{\substack{h,j,k\in\mathbb{Z}^2\setminus\{0\}\\h+j+k=0\\|h|>|j|}}\cdots\\
         =:& I_1 + I_2.
	\end{align*}

For $I_1$, it is easy to check that $\psi(x) = [\ln(c+x)]^2-[\ln(d+x)]^2$ is decreasing for $e\leq d < c$ so that 
\begin{equation*}
	e^{b(\ln(|k|+e))^2-b(\ln(|j|+e))^2} \leq e^{b(\ln(|h|+|j|+e))^2-b(\ln(|j|+e))^2}\leq e^{b(\ln(2|h|+e))^2-b(\ln(|h|+e))^2}.
\end{equation*}
Moreover, since 
	\begin{align*}
	(\ln(2|h|+e))^2-(\ln(|h|+e))^2 =& \ln \frac{2|h|+e}{|h|+e}\ln[(2|h|+e)(|h|+e)]\\
	\leq& \ln2[\ln2+2\ln(|h|+e)],
	\end{align*}
	it follows that
	\begin{align}
	\label{e57}
		I_1 =& L^2\sum_{\substack{h,j,k\in\mathbb{Z}^2\setminus\{0\}\\h+j+k=0\\|h|\leq |j|}}(\widehat{u}(h)\cdot j)(\widehat{E^bu}(j)\cdot\hat{w}(k))e^{b(\ln(|k|+e))^2-b(\ln(|j|+e))^2}\\
		\leq&L^2\sum_{\substack{h,j,k\in\mathbb{Z}^2\setminus\{0\}\\h+j+k=0\\|h|\leq |j|}}(\hat{u}(h)\cdot j)(\widehat{E^bu}(j)\cdot\hat{w}(k))e^{b(\ln2)^2}(|h|+e)^{2b\ln2}\nonumber\\
		\leq&L^2\sum_{\substack{h,j,k\in\mathbb{Z}^2\setminus\{0\}\\h+j+k=0\\|h|\leq |j|}}(\hat{u}(h)\cdot j)(\widehat{E^bu}(j)\cdot\hat{w}(k))e^{b(\ln2)^2}(1+e)^{2b\ln2}|h|^{2b\ln2}\nonumber\\
		=& \frac{c_3c_A}{\ko^{2b\ln 2}}|A^{b\ln2}u|^{1/2}|A^{1+b\ln2}u|^{1/2}|A^{1/2}E^bu||w|.\nonumber
	\end{align}
	
	By estimating $I_2$ in the same way, just replacing the right hand side of the first equality of (\ref{e57}) by 
\begin{equation}	
	L^2\sum_{\substack{h,j,k\in\mathbb{Z}^2\setminus\{0\},\\h+j+k=0,\\|j|\leq |h|}}(\widehat{E^bu}(h)\cdot j)(\widehat{u}(j)\cdot\hat{w}(k))e^{b(\ln(|k|+e))^2-b(\ln(|h|+e))^2},
\end{equation}	 
	 we obtain
	\begin{align}
	\label{e58}
		I_2 \leq& \frac{c_3c_A}{\ko^{2b\ln 2}}|A^{1/2+b\ln2}u|^{1/2}|A^{3/2+b\ln2}u|^{1/2}|E^bu||w|.
	\end{align}
	
	Since (\ref{e57}) and (\ref{e58}) are true for arbitrary $w$, then we infer (\ref{e55}).
\end{proof}

Using interpolation, it is easy to obtain the following estimates on $A^\alpha u$ for any $\alpha>0$ in the strip $\mathcal{S}(\delta)$. 
\begin{lem}
\label{l56}
Suppose $\alpha \geq 3$, and $\alpha \in [\frac{\gamma}{2}, \frac{\gamma+1}{2})$. Then 
\begin{equation}
	|A^{\frac{\alpha}{2}} u| \leq \Rt{\alpha}\nu\ko^\alpha,
\end{equation}
where 
\begin{equation}
\label{e510}
	\Rt{\alpha}^2 = C(g)\beta_1^{4^\gamma(6\alpha+1-3\gamma)}\beta_2^{-\gamma^2+(4\alpha-1)\gamma+11\alpha},
\end{equation}
and $C(g), \beta_1, \beta_2$ are defined in Theorem \ref{t64}.
\end{lem}
\begin{proof}
	By interpolation, we have
	\begin{equation}
		|A^\alpha u| \leq |A^{\frac{\gamma}{2}}u|^{\gamma+1-2\alpha}|A^{\frac{\gamma+1}{2}}u|^{2\alpha-\gamma}.
	\end{equation}
		Using Theorem \ref{thm33}, it follows that
	\begin{align*}
		|A^\alpha u|^2 \leq& [C(g)\beta_1^{4^{\gamma}}\beta_2^{\gamma^2+ \frac{9}{2}\gamma}]^{\gamma+1-2\alpha}[C(g)\beta_1^{4^{\gamma+1}}\beta_2^{(\gamma+1)^2+ \frac{9}{2}(\gamma+1)}]^{2\alpha-\gamma}\nu^2\ko^{2\alpha}\\
		=& C(g)\beta_1^{4^\gamma(6\alpha+1-3\gamma)}\beta_2^{-\gamma^2+(4\alpha-1)\gamma+11\alpha}\nu^2\ko^{2\alpha}\\
		=& \Rt{\alpha}^2\nu^2\ko^{2\alpha}.
	\end{align*}
\end{proof}

Using Lemma \ref{l55} and Lemma \ref{l56}, we are ready to obtain an estimate of $|u|_b$ in the strip $\mathcal{S}(\delta)$.
\begin{thm}
\label{t63}
	If $0\in\mathcal{A}$ and if $u(t)$, $t\in \mathbb{R}$ is any solution of the NSE in $\mathcal{A}$, then $u(t)$ satisfies 
\begin{equation}
\label{rnew}
|u(\zeta)|_b \leq \Rt{new}\nu,   \quad\forall\, \zeta \in \mathcal{S}(\delta),
\end{equation} 
where
\begin{equation}
\label{rtnew}
	\Rt{new}^2 := e^{M_1\delta\nu\ko^2}(\frac{4}{3}c_2\Rt{1}^2+c_0^{\frac{1}{3}}(c_1\Rt{1})^{\frac{2}{3}})+M_2(e^{M_1\delta\nu\ko^2}-1),
\end{equation}
\begin{equation}
	M_1 := 4(c_3c_A)^2\Rt{3+2b\ln2}^2+2\sqrt{2}c_3c_A\Rt{3+2b\ln2},
\end{equation}
and 
\begin{equation}
	M_2 := \frac{2\sqrt{2}{|A^{-\frac{1}{2}}E^bg|^2}}{(2\sqrt{2}(c_3c_A)^2\Rt{3+2b\ln2}^2+2c_3c_A\Rt{3+2b\ln2})\nu^4\ko^2}.
\end{equation}
\end{thm}
\begin{proof}
	It is shown in Remark 7.3 in \cite{FJLRYZ13} that if $0 \in \mathcal{A}$, then there exists a constant $\beta_3$ such that 
	\begin{equation}
		u(t_0)\in \mathcal{C}(\frac{3}{2}\ln\beta_3), \quad \forall\, t_0\in\mathbb{R},
	\end{equation}
	Applying Corollary \ref{c54} and Proposition \ref{t53}, we get that 
	\begin{equation}
		u(t_0) \in E_{(96\ln\beta_3)^{-1}},
	\end{equation}
	and
	\begin{align}
	\label{e516}
		|E^bu(t_0)|^2 \leq& \frac{4}{3}c_2|u|^2+c_0^{\frac{1}{3}}(c_1|A^{\frac{1}{2}}u|)^{\frac{2}{3}}\nu^{\frac{4}{3}}\ko^{-\frac{2}{3}}\\
		\leq& \frac{4}{3}c_2\Rt{1}^2\nu^2+c_0^{\frac{1}{3}}(c_1\Rt{1})^{\frac{2}{3}}\nu^2,\nonumber
	\end{align}
	where 
	\begin{equation}
		c_2 := c(\epsilon)|_{\epsilon=1} = e^{2b(\ln(e^2+e))^2}.
	\end{equation}
Taking the inner product of (\ref{cnse}) with $E^{2b}u$ we obtain 
\begin{equation*}
	\frac{1}{2}\frac{d}{d\rho}|E^bu(t_0+\rho e^{i\theta})|^2+\nu \cos\theta |A^{1/2}E^bu|^2 \leq |A^{\frac{1}{2}}E^bu||A^{-\frac{1}{2}}E^bg|+|(E^bB(u,u), E^bu)|,
\end{equation*}
where $\theta\in [-\pi/4, \pi/4]$.
Using Lemma \ref{l55}, we have that
\begin{align*}
	\frac{1}{2}\frac{d}{d\rho}|E^bu(t_0&+\rho e^{i\theta})|^2+\nu \cos\theta |A^{1/2}E^bu|^2 \leq \frac{c_3c_A}{\ko^{2b\ln 2}}|A^{b\ln2}u|^{1/2}|A^{1+b\ln2}u|^{1/2}|A^{1/2}E^bu||E^bu|\\
	&+\frac{c_3c_A}{\ko^{2b\ln 2}}|A^{1/2+b\ln2}u|^{1/2}|A^{3/2+b\ln2}u|^{1/2}|E^bu|^2+|A^{\frac{1}{2}}E^bu||A^{-\frac{1}{2}}E^bg| \;.
\end{align*}
By Young's inequality, we get
\begin{align*}
	\frac{d}{d\rho}|E^bu(t_0+\rho e^{i\theta})|^2&+\nu \cos\theta |A^{1/2}E^bu|^2 \leq \frac{2}{\nu\cos\theta}(\frac{c_3c_A}{\ko^{2b\ln 2}})^2|A^{b\ln2}u||A^{1+b\ln2}u||E^bu|^2\\
	&+\frac{2c_3c_A}{\ko^{2b\ln 2}}|A^{1/2+b\ln2}u|^{1/2}|A^{3/2+b\ln2}u|^{1/2}|E^bu|^2+2\frac{|A^{-\frac{1}{2}}E^bg|^2}{\nu\cos\theta} \\
	\leq& \eta_1 + \tilde{\eta}_2|E^bu|^2,
\end{align*}
where 
\begin{equation}
\label{e519}
	\eta_1 = \frac{2\sqrt{2}}{\nu}{|A^{-\frac{1}{2}}E^bg|^2}\;,
\end{equation}
and
\begin{equation*}
	\tilde{\eta}_2 := \frac{2\sqrt{2}}{\nu}(\frac{c_3c_A}{\ko^{2b\ln 2}})^2|A^{b\ln2}u||A^{1+b\ln2}u|+\frac{2c_3c_A}{\ko^{2b\ln 2}}|A^{1/2+b\ln2}u|^{1/2}|A^{3/2+b\ln2}u|^{1/2}.
\end{equation*}

For $\tilde{\eta}_2$, applying the Poincar\'{e} inequality and Lemma \ref{l56}, we have the following estimate
\begin{align}
\label{e521}
	\tilde{\eta}_2 \leq& \frac{2\sqrt{2}}{\nu}(\frac{c_3c_A}{\ko^{2b\ln 2}})^2\frac{1}{\ko^4}|A^{3/2+b\ln2}u|^2+\frac{2c_3c_A}{\ko^{2b\ln 2}}\frac{1}{\ko}|A^{3/2+b\ln2}u|\\
	\leq& \frac{2\sqrt{2}}{\nu}(\frac{c_3c_A}{\ko^{2b\ln 2}})^2\frac{1}{\ko^4}\Rt{3+2b\ln2}^2\nu^2\ko^{6+4b\ln2}+\frac{2c_3c_A}{\ko^{2b\ln 2}}\frac{1}{\ko}\Rt{3+2b\ln2}\nu\ko^{3+2b\ln2}\nonumber\\
	=& \left[2\sqrt{2}(c_3c_A)^2\Rt{3+2b\ln2}^2+2c_3c_A\Rt{3+2b\ln2}\right]\nu\ko^2=: \eta_2\;,\nonumber
\end{align}
where  $\Rt{3+2b\ln2}$ is defined as in (\ref{e510}).

Then we have
\begin{equation}
	\frac{d}{d\rho}|E^bu(t_0+\rho e^{i\theta})|^2
	\leq \eta_1 + {\eta}_2|E^bu|^2.
\end{equation}
It follows from Gronwall's inequality that
\begin{equation}
\label{e520}
	|E^bu(t_0+\rho e^{i\theta})|^2 \leq e^{\sqrt{2}\delta\eta_2}|E^bu(t_0)|^2+\frac{\eta_1}{\eta_2}(e^{\sqrt{2}\delta\eta_2}-1).
\end{equation}

Plugging (\ref{e516}), (\ref{e519}), (\ref{e521}) into (\ref{e520}), we obtain that
\begin{equation*}
	|E^bu(t_0+\rho e^{i\theta})|^2 \leq e^{M_1\delta\nu\ko^2}\left[\frac{4}{3}c_2\Rt{1}^2+c_0^{\frac{1}{3}}(c_1\Rt{1})^{\frac{2}{3}}\right]\nu^2+M_2(e^{M_1\delta\nu\ko^2}-1)\nu^2 = \Rt{new}^2\nu^2,
\end{equation*}
where $\Rt{new}$ is defined in (\ref{rtnew}).
Since $t_0\in\mathbb{R}$ and $\theta\in [-\pi/4, \pi/4]$ are arbitrary, it follows that 
(\ref{rnew}) holds for all $\zeta\in \mathcal{S}(\delta)$.
\end{proof}

\section{An Explicit Criterion}

In section 6, we found that generically $0$ is not in the attractor $\mathcal{A}_g$ since if $0\in \mathcal{A}_g$, then $g$ must be in the set $E_b$ which is of first category (See Theorem \ref{t73}). One immediately asks the following question: if $g \in E_b$, will $0\in \mathcal{A}_g$?  We partially answer this question by presenting a concrete criterion that is both sufficient and necessary for $0 \in \mathcal{A}_g$.

To present our result, we need some preparation. First, Theorem \ref{t63} tells us we can choose $\delta>0$ and $M>0$, such that for every $u_0\in \mathcal{A}$, $S(t)u_0$ is extendable to a holomorphic function on $\mathcal{S}(\delta)=\{z\in \mathbb{C}:|\Im z|<\delta\}$ with values in $E_b$, and $|S(t)u_0|_b\leq M$ for all $t \in \mathcal{S}(\delta)$.

Let $u_0=0 \in \mathcal{A}$; let $u(t)=S(t)u_0$ be the solution of the NSE; we use the conformal mapping (see \cite{FJK96})
\begin{align*}
\phi: \mathcal{S}(\delta) \rightarrow \Delta=\{T \in \mathbb{C}: |T|<1\}
\end{align*}
defined by the following formula
\begin{align*}
T=\phi(t)=\frac{\exp({\pi t}/{2 \delta})-1}{\exp({\pi t}/{2 \delta})+1}, \quad  t\in \mathcal{S}(\delta)
\end{align*}
with inverse given by
\begin{align*}
t=\phi^{-1}(T)=\frac{2\delta}{\pi}[\log(1+T)-\log(1-T)].
\end{align*}
The function $U(T)=u(t)$ satisfies the ODE
\begin{equation}
\label{ode4U}
 \frac{dU}{dT}=\delta_0\psi(T)\{g-\nu AU-B(U,U)\},    \quad T\in \Delta  \\ 
 \end{equation}
with initial value 
\begin{align*}
U(0)=u_0
\end{align*}
where 
\begin{align*}
\psi(T)=\frac{1}{2}\left( \frac{1}{1+T}+\frac{1}{1-T}\right)= \frac{1}{1-T^2}
\end{align*}
and $\delta_0={4\delta}/{\pi}$. 

By the analyticity of the function $U(T)$, we may express it in a Taylor series
\begin{equation}
\label{taylor}
U(T)=U_0+U_1T+U_2T^2+\cdots\;.
\end{equation}
Note that $U_0=u_0$. The convergence radius of the series (\ref{taylor}) is at least 1 if $u_0\in \mathcal{A}$, and it may be less than 1 if $u_0 \notin \mathcal{A}$.

Combining the series expansion form (\ref{taylor}) for $U(T)$  and the ODE (\ref{ode4U}), we get
\begin{align*}
\frac{d}{dT}(\sum_{n=0}^{\infty}U_nT^n)=\frac{\delta_0}{1-T^2}\left[g-\nu A \sum_{n=0}^{\infty}U_nT^n-
\sum_{n=0}^{\infty}\sum_{h+k=n}B(U_h, U_k)\right]
\end{align*} 
from which we get the following criterion for $0 \in \mathcal{A}_g$.
\begin{thm}
\label{theorem82}
$0 \in \mathcal{A}_g$ if and only if the Taylor series
\begin{equation}
\label{criterion series}
\sum_{n=0}^{\infty}U_n T^n, \quad T\in \Delta
\end{equation}
converges in $|\cdot|_b$ for all $T \in \Delta$ and the sum $U(T)=\sum_{n=0}^{\infty}U_n T^n$, for $|T|<1$, satisfies an estimate $|U(T)|_b\leq M$, for some $M>0$, where $U_n$ are computed recursively according to

\begin{align*}
U_0=0, \quad U_1=\delta_0 g, \quad U_2=-\frac{\nu \delta_0^2}{2}Ag \\
\end{align*}
and for $n\geq 2$
\begin{align}
\label{rec-relation}
U_{n+1}=\frac{n-1}{n+1}U_{n-1}-\frac{\nu \delta_0}{n+1}AU_n 
-\frac{\delta_0}{n+1}\sum_{\substack{h+k=n\\  h, k \geq1}}  B(U_k,U_h).
\end{align}
\end{thm}

\begin{remark}
Several remarks are in order.

1. Notice that all the $U_n$'s defined in the Theorem \ref{theorem82} depend only on g.

2. The application of the criterion given in Theorem \ref{theorem82} does not seem to be an easy 
task in general. We illustrate its use in the next section in the special case of forcing a single eigenvector of $A$. 
\end{remark}

\section{The case of Kolmogorov forcing}

We now use the criterion given in Theorem \ref{theorem82} to show that if the force $g \neq 0$ is an eigenvector of the Stokes operator $A$, with corresponding eigenvalue $\lambda > 0$, then 0 cannot be in $\mathcal{A}_g$.

If $0 \in \mathcal{A}_g$, where $Ag=\lambda g$, then noting that $B(g,g)=0$, the following lemma immediately follows from the the recursive relation (\ref{rec-relation}) given in Theorem \ref{theorem82}.

\begin{lem}
\label{ap_lem1}
For the coefficients $U_n$, we have 
\begin{equation*}
U_n=p_n(\lambda)g,  \quad n=1, 2, 3, \cdots, 
\end{equation*}
where $p_n(\cdot)$ are polynomials satisfying the following relations:
\begin{align}
	p_1(\lambda) =& \delta_0,\label{p1}\\
	p_2(\lambda) =& -\frac{\nu}{2} \lambda \delta_0^2,\label{p2}\\
	p_{N+1}(\lambda) =& \frac{N-1}{N+1}p_{N-1}(\lambda)-\frac{\nu \delta_0 \lambda}{N+1}p_{N}(\lambda), N=2, 3, \cdots\label{pn}.	
\end{align}
\end{lem}
\begin{proof}
By  Theorem \ref{theorem82}, we can obtain (\ref{p1}) and (\ref{p2}) easily.
Assume by induction that $U_n=p_n(\lambda)g$ is valid for all $n \leq N$, where $N\geq 2$. Then by (\ref{rec-relation}),
\begin{align*}
(N+1)U_{N+1} &=(N-1)U_{N-1}-\nu \delta_0 AU_N-\delta_0 \sum_{h+k=N}B(U_k, U_h)\\
 &=(N-1)p_{N-1}(\lambda)g-\nu \delta_0 \lambda p_N(\lambda)g-\delta_0 \sum_{h+k=N}p_h(\lambda)p_k(\lambda)B(g,g)\\
     &=(N-1)p_{N-1}(\lambda)g- \nu \delta_0 \lambda p_N(\lambda)g.\\
\end{align*}
Therefore,
\begin{equation*}
U_{N+1}=p_{N+1}(\lambda)g,
\end{equation*}
where,
\begin{equation*}
p_{N+1}(\lambda)=\frac{N-1}{N+1}p_{N-1}(\lambda)-\frac{\nu \delta_0 \lambda}{N+1}p_{N}(\lambda).
\end{equation*}
The proof is completed by the induction hypothesis.
\end{proof}

From the above lemma and Theorem \ref{theorem82}, we conclude that if $0 \in \mathcal{A}_{g}$, then the solution $u(t)$ is of a special form, namely, $u(t)=\phi(t)g$, where $\phi(t)$ is a bounded real-valued function on 
$\mathbb{R}$. Clearly the function $\phi(t)$ must satisfy the following ODE:
\begin{equation*}
\frac{d\phi}{dt}+\nu \lambda\phi=1,
\end{equation*}
from which it follows that
\begin{equation*}
\phi(t)=\frac{1}{\nu \lambda}+(\phi(0)-\frac{1}{\nu \lambda})e^{-\nu \lambda t}.
\end{equation*}

Boundedness of the solution $u(t)$ for all negative time implies that $\phi(0)=\frac{1}{\nu \lambda}$, and hence $u(t)\equiv \frac{g}{\nu \lambda}$. This contradicts $u(0)=\phi(0)g=0$.
Therefore, in this case, using the criterion and dynamics analysis, we obtain that $0$ is not in $\mathcal{A}_{g}$.
\section{Appendix}
The bounds in Theorem \ref{t64} are found recursively, starting with 
\begin{equation*}
\Rt{1} = \sqrt{2}G\;, \quad \Rt{2} = \left(\frac{3(\sqrt{2}\cdot 16^2\cdot 24^6c_L^{16})^{2/3}}{4(2c_L^2+c_A)^{4/3}}G^6+4 R_{2}^2\right)^{\frac{1}{2}}\;, \quad \Rt{3} = 4\frac{N_2^{\frac{1}{2}}}{\nu^{\frac{1}{2}}\ko\delta_3^{\frac{1}{2}}},
\end{equation*} 
where
\begin{equation*}
R_2 = 2137c_L^4G^3\;, \quad 
N_2=R_{2}^2+\frac{2\delta_2\Rt{1}^2}{\nu\ko^2\delta_1^2}+16(2c_L^2+c_A)^2\nu\ko^2\delta_2\Rt{1}\Rt{2}^3.
\end{equation*}
The  width of the strip of analyticity is estimated from below in each case as
\begin{align*}
\delta_1 = \frac{1}{16\cdot 24^3 c_L^8 \nu \ko^2 G^4}\;,
\end{align*}
\begin{align*}
\delta_2 &= \min\left\{\delta_1, 16^{-1}\left[ (2c_L^2+c_A)^{\frac{8}{3}}\Rt{1}^{\frac{8}{3}}\left(\frac{\nu\ko^2}{8\delta_1^2}\right)^{\frac{2}{3}}+(2c_L^2+c_A)^4(\nu\ko^2)^2\Rt{1}^2 R_2^2\right]^{-\frac{1}{2}}\right\}, \quad \delta_3=\frac{\delta_2}{2}\;.
\end{align*}
The other constants are given by
	\begin{align*}
C(g)=C_1C_2\Rt{3}^2\beta_2^{-19/2}\;,\quad	\beta_1=e^{2\sqrt{2}\nu\ko^2C_3\delta_\alpha}\;, \quad \beta_2=\max\{\frac{72\sqrt{2}}{\pi^2},c_A^2\Rt{1}\Rt{2}\}\;, 
	\end{align*}
where
	\begin{align*}
	C_1:=\prod_{\gamma=3}^{\infty}(1+\epsilon_\gamma)\;,\quad
	C_2:=3^32^{-7}c_L^8\Rt{1}^2\prod_{\gamma=3}^{\infty}(1+\eta_\gamma)
	\;, \quad C_3:={4}\left(2^{\frac{5}{2}}c_A^2\Rt{1}\Rt{2}+2^{\frac{1}{2}}c_A\sqrt{\Rt{1}\Rt{3}}\right)	\;,
	\end{align*}
	\begin{equation*}
\epsilon_\gamma=\frac{1}{2\sqrt{2}\Gamma_{\gamma}\delta_\alpha\nu\ko^2}+\frac{\sqrt{2}}{\Gamma_{\gamma}\nu^2\ko^4\delta_\alpha^2}+\frac{\pi^2}{72\nu^2\ko^4\delta_\alpha^4\Gamma_{\gamma}\Gamma_{\gamma+1}}\;,
 \quad \eta_\gamma=\frac{\sqrt{\Rt{1}\Rt{3}}}{2^{\gamma+2}c_A\Rt{1}\Rt{2}},		
	\end{equation*}
and 
	\begin{equation*}
	\Gamma_\gamma:=2^{\gamma+\frac{3}{2}} c_A\left(2^{\gamma+2}c_A\Rt{1}\Rt{2}+\sqrt{\Rt{1}\Rt{3}}\right).
	\end{equation*}


\end{document}